\documentclass[12pt, reqno]{amsart}
\usepackage{eurosym}
\usepackage{amsmath,amssymb,mathtools}
\usepackage{amsmath,amssymb,mathtools,amsthm,
	textcomp,mathscinet}
\usepackage{amsfonts,graphicx,enumerate,subcaption,placeins}
\usepackage{times}
\usepackage{chngcntr}

\numberwithin{equation}{section}
\newtheorem{theorem}{\bf Theorem}[section]
\newtheorem{remark}{\bf Remark}[section]

\newtheorem{lemma}{Lemma}[section]
\newtheorem{corollary}{Corollary}[section]
\newtheorem{example}{Example}[section]

\newcommand{\ncom}{\newcommand}
\ncom{\nno}{\nonumber}
\ncom{\nin}{\noindent}
\newcommand{\ds}{\displaystyle}

\begin{document}
	
	\hfill{\today}
	\vspace*{.6cm}
	
	\title{Explicit Formulas for the  Divergence Operator in Isonormal Gaussian Space}
	\author{ S. Levental and P. Vellaisamy}
	\address{Department of Mathematics, Indian Institute of Technology Bombay,
		Powai, Mumbai 400076, INDIA}
	\address{Department of Statistics and Probability, Michigan State
		University, East Lansing, MI 48824, USA.}
	\email{pv@math.iitb.ac.in, levental@.msu.edu}
	\thanks{This work was completed while the first author was visiting the
		Department of Statistics and Probability, Michigan State University, during
		Summer-2019.}
	\date{}
	
	\begin{abstract}
		In this paper, we  first derive some explicit formulas for the computation of the $n$-th order divergence operator in Malliavin calculus in  the one-dimensional case.  We then extend these results  to the case of isonormal Gaussian space. Our results
		generalize some  of the known results for the divergence operator.	Our approach in deriving the formulas is new and simple.
	\end{abstract}
	
	\maketitle

	\section{Introduction}
	
	\nin Malliavin calculus is an infinitesimal differential calculus on the Weiner space.
	It deals with the random elements, which are functions of possible infinite dimensional Gaussian field. The Malliavin derivative
	and the divergence operator, the adjoint of the derivative operator, form some of the main
	tools of the Malliavin calculus. The divergence operator can  be viewed also as
	a generalized stochastic integral or Skorohod integral which forms the basis of extending stochastic calculus from  adaptive to anticipating 
	one.
	
	\vspace{.3cm}
	\nin The main objective  of the paper is  to provide some explicit formulas for the  divergence operator of the $n$-th order $n$, in certain cases. First, we discuss the one-dimensional case,  in the spirit of  Chapter 1 of Nourdin and Peccati (2012), referred to hereafter as NP(2012). We then extend the results to the usual setup of isonormal Gaussian spaces, which are discussed in detail in several books; see for example, NP (2012) and Nualart (2009). For a recent introduction to
	this topic, see Nualart and Nualart (2018). 
	
	\vspace{.3cm}
	\nin The divergence operator is closely connected with Hermite polynomials. Therefore, we start  with some basic properties of Hermite polynomials that are used in the paper. Though most of these results are known, the proofs are rather different. For the one-dimensional setup, we derive an interesting and explicit formula for the divergence operator of the $n$-th order $n$ that is acting on smooth functions. This new result appears in Theorem 3.1 which serves as a basis to proving all the results that follow. Another result, Theorem 3.2, presents an interesting alternative formula to that of Theorem 3.1. It seems that Theorem 3.2 formulation cannot be achieved without the help of Theorem 3.1. Also, using Theorem 3.1, we prove the analogous result for the usual setup of Malliavin calculus, the so-called isonormal Gaussian spaces. The outcome is Theorem 4.1 which generalizes a fundamental result in the literature, as illustrated in Example 4.1. Another concrete formula, that is new and 
	concerns the Malliavin derivative and the divergence operator,  is Corollary 4.1.  A special case of Corollary 4.1 is similar to a well-known result
	in the literature.

	\vspace{.3cm}
	\nin The paper is organized as follows: In Section 2, some basic properties of Hermite polynomials are presented with proofs that are based on our approach. In Section 3, we deal with the one-dimensional case with the high point being the explicit formula that is presented in Theorem 3.1. For the one-dimensional case, this result leads to some additional
	and interesting formulas for the divergence operator and also a result that connects Hermite polynomial with the standard normal variate.  In Section 4, we deal with the extension of the main result of Section 3 to isonormal Gaussian spaces. This main result  generalizes a result which is considered as an important result  in the literature.\\

\section{Some basic Results on Hermite Polynomials}

In this section, we briefly discuss  Hermite polynomials and their properties, as they
are closely connected with the divergence operator.
\noindent The Hermite polynomials, denoted by $H_n(x), x \in R, $ is defined as
\begin{align}
H_n(x)=& (-1)^n e^{x^2/2} D^n  e^{-x^2/2},  
\end{align}
where $D= \frac{d}{dx}.$ 
Indeed, the above definition of Hermite polynomials is called Rodriques' formula. We express  $H_n(x)$ using the 
standard normal density, as this approach is easier to study
its properties. Let $N(0,1)$ denote the standard normal distribution. Then
\begin{align} \label{212}
H_n(x)=&(-1)^n \sqrt{2\pi}e^{x^2/2} \phi^{(n)}(x) \nno \\
=& (-1)^n \frac{\phi^{(n)}(x)}{\phi (x)},
\end{align}
where $\phi(x)$ denotes the density of the random variable $N \sim N(0, 1)$. Here and henceforth, $ \phi^{(0)}(x)= \phi(x)$ and $ \phi^{(n)}(x)$ denotes the $n$-th derivative of  $\phi(x)$. Also, it can easily be checked that $ \phi^{(1)}(x)
= \phi^{'}(x)= -x \phi(x).$ This fact is used often in proving our results related to Hermite polynomials.

\vspace{.3cm}
 It follows easily from \eqref{212} that
$ H_0(x)=1,~H_1(x)=x, ~H_2(x)=x^2-1,~ H_3(x)=x^3-3x, ~H_4(x)=x^4-6x^2+3$,~
and etc.

\vspace{.3cm}
First we state and prove some important properties of $H_n(x)$, using our approach,
which will be useful later. Though most of the the results are known, our proofs and approach are new and different. We have also used  both the prime and the integer 1 in the superscript
 of a function to denote its first derivative. This is
 used as per the context and to  avoid any further confusion.

\begin{lemma}\label{lem21}
Let   $H_{n}(x)$ be $n$-th degree Hermite polynomial. Then the  following holds:
 For $ n \geq 1$,
\begin{enumerate}
\item[(i)] We have
\begin{align} \label{213}
 H_{n+1}= H_1H_n-nH_{n-1}.
\end{align}	
 \item[(ii)] The derivative $H_n^{'}(x)$ of $H_n(x)$ satisfies
 \begin{align} \label{214}
H_{n}^{'}=&   H_1 H_n-H_{n+1} 
    = nH_{n-1}.
 \end{align}

 \item[(iii)] The exponential generating function of $H_n(x)$ is
 \begin{align} \label{215}
 G_{H}(t,x)=& \sum_{n=0}^{\infty} \frac{t^n}{n!} H_n(x)  =\frac{\phi(t-x)}{\phi(x)}.
 \end{align}
\end{enumerate}
\end{lemma}

\begin{proof} (i) It can be seen that $\phi$ satisfies  the recurrence relation, for $n \geq 1$,
	\begin{align} \label{216}
		\phi^{(n)}= - \displaystyle \Big[x\phi^{(n-1)}+ (n-1)\phi^{(n-2)}\Big].
	\end{align}
	
	\noindent  Using the above relation, we get
	\begin{align*}
		(-1)^n \phi^{(n)}=& (-1)^{n+1 }\displaystyle \Big[x\phi^{(n-1)}+ (n-1)\phi^{(n-2)}\Big] \\	
		=& \Big[(-1)^{2}(-1)^{n-1}x\phi^{(n-1)}+ (-1)^{3}(-1)^{n-2} (n-1)\phi^{(n-2)}\Big]\\
		= & \left[(-1)^{n-1}x\phi^{(n-1)}-(-1)^{n-2} (n-1)\phi^{(n-2)}    \right]
	\end{align*}
	which when divided by $\phi$ leads to, for $n \geq 2$,
	\begin{align} \label{217}
		H_{n}= \displaystyle \Big[ H_1 H_{n-1}-(n-1)H_{n-2}\Big].
	\end{align}
or equivalently, we have for $ n \geq 1$,	
\begin{align} \label{218}
H_{n+1}= \displaystyle \Big[ H_1 H_{n}- n H_{n-1}\Big].
\end{align}	
	 (ii). From \eqref{212},
 \begin{align*} 
H_n^{'}(x)=& (-1)^n \Big[\frac{\phi^{(n+1)}(x)}{\phi (x)}- \frac{\phi^{(n)}(x)}{\phi^2 (x)} \phi^{'}(x)\Big] \nonumber \\
=& (-1)^n \Big[\frac{\phi^{(n+1)}(x)}{\phi(x)}+x \frac{\phi^{(n)}(x)}{\phi(x)} \Big]   \nonumber \\
=&   H_1(x) H_n(x)-H_{n+1}(x).
\end{align*}

\nin Thus, we get
\begin{align} \label{219}
H_n^{'}
=&   H_1 H_n-H_{n+1}.
\end{align}
Also, using \eqref{213} and \eqref{219},
\begin{align} \label{210}
H_n^{'}
  =&  H_1 H_n-\{H_1H_n-nH_{n-1} \} \nonumber \\
  =& nH_{n-1}.
\end{align}

\nin (iii) The Taylor series expansion of $\phi(t)$ about $x$ 
leads to
\begin{align*}
\phi(t-x)=& \phi(x)-t \phi^{(1)}(x)+ \frac{t^2}{2!}\phi^{(2)}(x) + \cdots +
(-1)^n \frac{t^n}{n!}\phi^{(n)}(x)+ \cdots\\
       =& \sum_{n=0}^{\infty}(-1)^n \frac{t^n}{n!}\phi^{(n)}(x). 
\end{align*}

\nin Hence,
\begin{align*}
\frac{\phi(t-x_)}{\phi(x)}=& 
 \sum_{n=0}^{\infty}(-1)^n \frac{t^n}{n!}\frac{\phi^{(n)}(x)}{\phi(x)} \hspace{6cm} \\
 =& \sum_{n=0}^{\infty} \frac{t^n}{n!} H_n(x)\\
=& G_{H}(t,x),
\end{align*}
as claimed.
\end{proof}

\noindent We start with a result that connects $H_n$ with the standard normal
variate $N$.
\begin{lemma} \label{nlem22} let $N$ be a standard normal variate and $
	H_n(x)$ be the $n$-th degree  Hermite polynomial.
	Then, for $ n \geq 1$,
\begin{align} \label{211}
   H_n= E(H_1+ iN)^n,
\end{align}
where $ i= \sqrt{-1}$ and $E$ denotes the expectation operator.
\end{lemma}

\begin{proof}
 \noindent Let $K_n(x)= E(H_1(x)+ iN)^n=E(x+iN)^n $ and $\phi_X(t)$ denote the characteristic 
function of the rv $X.$ Then the exponential generating function of $K_n(x)$ is, for $t>0$, 
\begin{align*}
	G_K{}(t,x)=\sum_{n=0}^{\infty} \frac{t^n}{n!}K_n(x)
	=&\sum_{n=0}^{\infty} \frac{t^n}{n!}E(x+iN)^n   \hspace{4cm} \\
	=&\sum_{n=0}^{\infty} \frac{t^n}{n!} \int_{-\infty}^{\infty} (x+iy)^n \phi(y)dy\\
	=& 	\int_{-\infty}^{\infty} \sum_{n=0}^{\infty} \frac{(tx+ity)^n}{n!} \phi(y)dy\\	
	=& \int_{-\infty}^{\infty} e^{(tx+ity)} \phi(y)dy\\
	=& e^{tx} \phi_N(t)\\
	=& e^{tx-t^2/2}\\
	=& \frac{\phi(t-x)}{\phi(x)},
\end{align*}
which is the exponential generating function of $H_n(x)$ (see \eqref{215}).

\nin As the generating functions of $H_n(x)$ and $K_n(x)$
agree for all $t >0$, we have $H_n(x)=K_n(x)$ for all $x \in R.$
\end{proof}

\nin An application of Lemma \ref{nlem22} shows that
\begin{align} \label{212a}
H_n(0)= E(iN)^n= ( -1)^{\frac{n}{2}}E(N)^n =
\begin{cases}
\ds 0,  &~\text{if}~ n~  \text{is ~odd},\\ 
\ds \frac{(-1)^{\frac{n}{2}} (n)!}{2^{\frac{n}{2}}(\frac{n}{2})!}, &~\text{if}~ n~  \text{is ~even}.
\end{cases}
\end{align}

\section {Divergence Operator in one-dimensional Case}
\vspace*{.3cm}
\noindent First we introduce some  notations and definitions (see  NP (2012), Chapter 1). Let $L^{2}(\phi)$ denote the set of square integrable functions with respect to the standard normal density $\phi$. Let  $ S \subset C^{\infty}
({R})$  denote the set of real-valued functions that have, together
with all their derivatives, polynomial growth.

\noindent The $n$-th divergence operator  $\delta^{n}: Dom (\delta^{n})
\to L^{2}(\phi) $ is such that for each  $ g \in Dom (\delta^{n}) $, the function $\delta^n g$ satisfies
\begin{align}\label{311}
\int f^{(n)}(x) g(x) \phi (x) dx= \int f(x) (\delta^{n} g)(x) \phi (x) dx, ~\mbox{for ~ all}~
f \in S,
\end{align}
where $ Dom (\delta^{n}) \subset  L^{2}(\phi) $ is defined by
\begin{align*}
Dom (\delta^{n})=&\Big \{g \in  L^{2}(\phi):  \int f^{n}(x)g(x) \phi(x)dx \leq c \sqrt{\int f^2(x) \phi(x)dx }\Big \}, 
\end{align*}
for  all $ f \in S$ and some constant $c = c(g)>0$.

\begin{remark}\label{rem31} \em{
	{\bf Notational Convention:} In the rest of the paper, the symbol $H$ stands for the generic Hermite polynomial. Also, we will follow
	the following rule for the binomial theorem expansion of $(H-g)^n $. We denote
	$H^0=H_0=1, H^k= H_k, k \geq 1,$ and for the real-valued function $g$, we denote
	$ g^0=g, g^k= g^{(k)}, k \geq 1,$ the  $k$-th derivative of $g$.	
 For example,
  \begin{align*}
 (H-g)^1= (H^1g^0-H^0g^1)=(H_1g-H_0g^{1})=(H_1g-g^{1}),
 \end{align*}
 and for $ x \in R$,
 \begin{align*}
   (H-g)^1(x)=(H_1g-g^{1})(x)=H_1(x)g(x)-g^{1}(x).
 \end{align*}
}
\end{remark}

\begin{remark} \em{ In fact, \eqref{311} holds for  $ f \in D^{n,2}$, where $  D^{n,2}
\supset S $ is a Sobolev space defined appropriately, see (1.1.6) in NP (2012). That is,  $ D^{n,2}$ is the closure of $S$ with respect to the norm
 \begin{align*}
\left\| f \right\|_{ D^{n,2}}	= \Big(\int_{R}|f(x)|^2 \phi(x)dx + \int_{R}|f^{'}(x)|^2 \phi(x)dx +  \cdots +  \int_{R}|f^{(n)}(x)|^2 \phi(x)dx \Big)^{1/2}.
\end{align*}
}
\end{remark}

\noindent It follows that $\delta^{0}g=g, g \in Dom (\delta^{0})= L^{2}(\phi). $ Also, it follows, from equation (1.2.4) of NP (2012), that if $ f \in Dom (\delta^{n+m})$, then $ \delta^{n} f
\in Dom (\delta^{m})$ and 
\begin{align} \label{312}
\delta^{n+m}f= \delta^{m} (\delta^n  f).
\end{align}
\nin It is known  (see equation (1.2.5) of NP (2012)) that the operator $\delta^ 1$ satisfies 
\begin{align} \label{313}
\delta^{1}g=\delta g = xg- g^{1}= H_1g-g^1, ~~ g \in D^{1,2},
\end{align}
since $H_1(x)=x$. Also, we wirte, for example, 
$\delta^{1}g$  as
\begin{align} \label{313a}
(\delta^1 g)(x) =  (H_1g-g^1)(x)= H_1(x)g(x)-g^1(x), ~~ x \in {R}.
\end{align} 
	
\noindent Using the result in \eqref{313}, one could obtain the expression for $\delta^{2}g $ which would involve $g$, $g^{1}$ and $g^{2}$ .   However, proceeding this way and  obtaining a general expression for $\delta^n g$ will be very complicated.
Further, an explicit formula for $ \delta^n$, as in \eqref{313} for the case $n=1$, is not 
available in the literature. 
 Our goal in this paper is to obtain some compact representations  and explicit computational formulas for
 $ \delta^n.$ First, we obtain it   for the  one-dimensional case, exploiting the connections to Hermite polynomials. We then  extend
 some of these results to case of isonormal Gaussian space.
 In the process, we generalize some of the known results for the divergence operator in the isonormal Gaussian space. \\

\noindent First, we state and prove the main result of this section. This result
provides a simple and compact representation formula for  $ \delta^n$ and its proof is combinatorial in nature.
\begin{theorem}\label{thm31}
Let  $\delta^n$ be the $n$-th divergence operator
defined in \eqref{311} and $g \in {S}$. Then
for $ n \geq 1$, we have
\begin{align} \label{314}
\delta^{n}g= (H-g)^n,
\end{align}
where the right-hand side is expanded using binomial theorem using the   rule
in Remark \ref{rem31}. 
\end{theorem}
\begin{proof} Note first  when $n=1$, we get from \eqref{314}, $ \delta^1 g= (H-g)^1= H_1g-H_0g^1$, which
	coincides with \eqref{313}.
Consider the case $n=2$. Using \eqref{312}, we get
\begin{align*}
\delta^{2}g=& \delta(\delta^1 g )\\
=&\delta \left(H_1g- g^1\right)\\
=& H_1\left(H_1g- g^1 \right)-\left(H_1g- g^1\right)^{1}  \hspace{0.4cm}( \text{using~ \eqref{313}})\\
= &(H_1^2-H_1^{'})g-2H_1g^1+g^2\\
=& H_2g-2H_1g^{1}+g^{2} \hspace{0.4cm} (\text{using~ \eqref{219}}) \\
=& (H-g)^2.
\end{align*}
which coincides with \eqref{314}.

\nin Assume  
$\delta^{k}g= (H-g)^k,$ for $ 1 \leq k \leq n$. We next
show that the result in \eqref{314} is true for $n+1$. 
Observe that
\begin{align*}
{\delta}^{n+1}g= & \delta (\delta^n g)\\
=&\delta(H-g)^n \\
=&\delta \Bigg(H_n g - \binom{n}{1}H_{n-1} g^1 + \cdots+ (-1)^k \binom{n}{k}H_{n-k} g^k
+\cdots \\
& \hspace{2cm}+ (-1)^{n-1} \binom{n}{n-1}H_{1} g^{n-1}+ (-1)^n g^n \Bigg)\\
=&H_1\Bigg(H_n g - \binom{n}{1}H_{n-1} g^1 + \cdots+ (-1)^k \binom{n}{k}H_{n-k} g^k
+\cdots \\
&\hspace{2cm} + (-1)^{n-1} \binom{n}{n-1}H_{1} g^{n-1}+ (-1)^n g^n \Bigg)\\
& -\Bigg(H_n g - \binom{n}{1}H_{n-1} g^1 + \cdots+ (-1)^k \binom{n}{1}H_{n-k} g^k
+\cdots  \\
&\hspace{2cm}+ (-1)^{n-1} \binom{n}{n-1}H_{1} g^{n-1}+ (-1)^n g^n \Bigg)^{'}\\
=& (H_1H_n-H_n^{'})g-\Bigg(\binom{n}{1}(H_1H_{n-1}-H_{n-1}{'})+H_n \Bigg)g^1 + \cdots  \\
& \hspace{2cm} +  (-1)^k \Bigg(\binom{n}{k}(H_1H_{n-1}-H_{n-1}{'})+\binom{n}{k-1}H_{n-k+1} \Bigg)g^k + \cdots \\
 & \hspace{2cm} +(-1)^n\Big(H_1+ (-1)^{n-1}n H_1 \Big)g^n + (-1)^{n+1}g^{n+1} \\
=& H_{n+1}g-\Bigg(\binom{n}{1}(H_1H_{n-1}-H_{n-1}^{'})+H_n \Bigg)g^1 + \cdots  \\
 & \hspace{2cm} +  (-1)^k \Bigg(\binom{n}{k}(H_1H_{n-k}-H_{n-k}^{'})+\binom{n}{k-1}H_{n-k+1} \Bigg)g^k + \cdots \\
& \hspace{2cm}  +(-1)^n\Big(H_1+ n H_1 \Big)g^n + (-1)^{n+1}g^{n+1} \\
=& H_{n+1}g-\Big(\binom{n}{1}H_n+H_n \Big)g^1+ \cdots+  \\  
& \hspace{2cm}  +  (-1)^k \Bigg(\binom{n}{k}H_{n-k+1} +\binom{n}{k-1}H_{n-k+1} \Bigg)g^k + \cdots \\
& \hspace{2cm} +(-1)^n(n+1)H_1g^n + (-1)^{n+1}g^{n+1} \\
=& H_{n+1}g-\binom{n+1}{1}H_n g^1 
+ \cdots +  (-1)^k \binom{n+1}{k}H_{n+1-k} g^k \\
+& \cdots +(-1)^n\binom{n+1}{1}H_1g^n + (-1)^{n+1}g^{n+1} \\
=& (H-g)^{n+1}.
\end{align*}
\nin Observe that we have used the following facts 
\begin{eqnarray*}
&&\,\,(i)\,\,\, H_1H_{n}-H_n^{'}= H_{n+1}, ~~\text{(see (ii) of Lemma
\ref{lem21})} \hskip 3.5in\\
&&(ii)\,\,\,  \binom{n}{k-1} + \binom{n}{k}= \binom{n+1}{k}.
\end{eqnarray*}
in the  steps above.

\nin Thus, we get, for $n \geq 1$,
\begin{align*}
\delta^{n}g= (H-g)^n,
\end{align*}
which proves the result.
\end{proof}

\nin When $ g \equiv 1$, we have  
\begin{align*}
  \delta^n 1=  
        (H-1)^n= H_n 1-0 =H_n,
\end{align*}
which is sometimes taken as the definition of $H_n$ (see p.~13 of NP (2012)).

\nin Next, we obtain  another expression for $\delta^n$, which could be helpful for computational purposes.

\begin{lemma}
	Let  $\delta^n$ be the $n$-th degree divergence operator
	defined in \eqref{311} and $g \in \mathcal{S}$. Then
	for $ n \geq 1$, we have
	\begin{align} \label{315}
		{\delta}^{n+1}g=&  \sum_{k=0}^{n} (-1)^{n-k} \binom{n}{k}  \Big(H_{k+1}g^{n-k}
		-H_kg^{n-k+1}  \Big).
	\end{align}
\end{lemma}

\begin{proof} Using Theorem \ref{thm31} and \eqref{312}, we get
\begin{align*}
{\delta}^{n+1}g= & \delta (\delta^n g)\\
=&\delta(H-g)^n \\
=&\delta \Big(  \sum_{k=0}^{n} (-1)^{n-k} \binom{n}{k}H_kg^{n-k}  \Big)\\
=&  \sum_{k=0}^{n} (-1)^{n-k} \binom{n}{k} \delta(H_kg^{n-k}) \\  
=&  \sum_{k=0}^{n} (-1)^{n-k} \binom{n}{k} \Big( H_1H_kg^{n-k}
- H_k^{'}g^{n-k}-H_kg^{n-k+1}  \Big) \\ 
=&  \sum_{k=0}^{n} (-1)^{n-k} \binom{n}{k}  \Big(H_{k+1}g^{n-k}
-H_kg^{n-k+1}  \Big),
\end{align*}
using \eqref{214}. This proves the result.
\end{proof}
\nin Using Theorem \ref{thm31}, we get the interesting following result which
generalizes Lemma  \ref{nlem22} and also the result in  \eqref{313}. Further,  it connects 
 $\delta^n $ with the moments of the standard normal distribution.\\
\begin{theorem}
Let $N$ denote the standard normal variable, $K= H_1+iN $, $g \in \mathcal{S}$
and $E$ denote the expectation is with respect to  $N$. Then,  for $n \geq 1$,
\begin{align*}
{\delta}^{n}g =& E(K-g)^n,
\end{align*}
 where the rhs is expanded using binomial theorem and the  rule
in Remark \ref{rem31} is applied only for $g$.
\end{theorem}

\begin{proof}
It follows, from Lemma \ref{nlem22}, that
$H_n= E(K^n)$, for $ n \geq 1.$ Also, by Theorem \ref{thm31}, we have for $n \geq 1$,
 \begin{align*}
 {\delta}^{n}g
 =&  \sum_{r=0}^{n} (-1)^{n-r} \binom{n}{r}H_rg^{n-r} \\
 =&  \sum_{k=0}^{n} (-1)^{n-r} \binom{n}{r} E(K^r)g^{n-r} \\  
 =& E \Big( \sum_{k=0}^{n} (-1)^{n-r} \binom{n}{r} K^r g^{n-r} \Big) \\
 =& E(K-g)^n,
 \end{align*}
which completes the proof.
\end{proof}

\begin{remark} {\em
 (i) When $g \equiv 1$, we have $ \delta^n= E(K-1)^n= E(K^ng^0-g^1+ \cdots)=E(K^n)=H_n$, which is Lemma  \ref{nlem22}.

(ii) When $n=1$, $\delta^1g= E(K-g)^1= E(K^1g^0-g^1)= E\left((H_1+iN)g- g^1\right)= H_1g-g^{1}$, which coincides with \eqref{313}. }
\end{remark}

\nin  Next, we give a different proof of the result in (1.3.5)  
of NP (2012, p.~12), using Theorem \ref{thm31}.
\begin{lemma} \label{lem32}
Let $D$ denote the usual derivative operator, $\delta^n$ be the $n$-th divergence operator and $g \in \mathcal{S}.$ Then
\begin{align} \label{316}
D {\delta}^{n}g =&  n {\delta}^{n-1}g +  {\delta}^{n}g^{1},
\end{align}
for $n \geq 1.$
\end{lemma}

\begin{proof} Using Theorem \ref{thm31}, we get
	\begin{align*}
	D{\delta}^{n}g =& D(H-g)^n \\
	=& D \Big(  \sum_{k=0}^{n} (-1)^{n-k} \binom{n}{k}H_kg^{n-k}  \Big)\\
   =&  \sum_{k=0}^{n} (-1)^{n-k} \binom{n}{k} D \left (H_kg^{n-k}\right) \\
   =&  \sum_{k=0}^{n} (-1)^{n-k} \binom{n}{k}  \left(H_k^{'}g^{n-k}+ H_kg^{n-k+1} \right) \\
   =&  \sum_{k=0}^{n} (-1)^{n-k} \binom{n}{k}  H_k^{'}g^{n-k}+ \sum_{k=0}^{n} (-1)^{n-k} \binom{n}{k} H_k  (g^{1})^{n-k} \\
  =&  \sum_{k=0}^{n} (-1)^{n-k} \binom{n}{k} k H_{k-1}g^{n-k}+  \delta^{n} g^{1} \hspace{.4cm} (\text{using}~ \eqref{210})\\
  = & n \sum_{j=0}^{n-1} (-1)^{n-j-1} \binom{n-1}{j}  H_{j}g^{n-1-j}+  \delta^{n} g^{1} \\
   = & n \delta^{n-1} g + \delta^{n} g^{1},
\end{align*}
which proves the result.
\end{proof}

\begin{remark}  \em When $n=1$, Lemma \ref{lem32} reduces to $D\delta g= g + \delta g^1$, which is Proposition 1.3.8 of NP (2012). Also, the formula in \eqref{316} can also
	be written as
\begin{align} \label{317}
\left(D {\delta}^{n} -{\delta}^{n}D\right) g =&  n {\delta}^{n-1}g 
\end{align}	
which is equation (1.3.5) of NP (2012, p.~12). Note the form in \eqref{316}  is similar to the derivative of the product of two functions. 
\end{remark}

\vspace{.2cm}
\section {Divergence Operator in  Isonormal Gaussian space}

Let $(\Omega, \mathcal{F}, P) $ be the underlying probability space, $ \mathcal{H}$ be a real separable Hilbert space and  $X=\{ X(h): h \in \mathcal{H} \}$
be the associated isonormal Gaussian process over $ \mathcal{H}.$ That is,
the   $  X(h)$ are zero mean normal rvs defined on
$ (\Omega, \mathcal{F}, P)$ such that  $ E(X(h) X(g))= \langle h, g \rangle_{\mathcal{H}}$, for  $ h, g \in \mathcal{H}.$
Assume $\mathcal{F} $ is generated by $X$ and let $L^2 (\Omega) = L^2 (\Omega, \mathcal{F}, P).$ It follows that $X: \mathcal{H} \to  L^2 (\Omega) $ is a
liner operator.

\vspace{.2cm}
Let  $f: \mathbb{R}^m \to \mathbb{R}$ be a $C^{\infty}$-function such that its partial derivatives have polynomial growth.  A random variable 
 $f(X(h_1), \cdots,  X(h_m))$,  $ h_i \in \mathcal{H}$, is called a smooth rv. Let $ \mathcal{S} $ be a collection  of all  smooth rvs.

\noindent Let  $ F \in \mathcal{S}$ be a smooth rv. The $n$-th Malliavin
derivative $D^n$ of $F$ is defined as 
\begin{align} \label{b41}	
 D^n F= \sum_{j_1, \cdots, j_n=1}^{m} \frac{\partial^n}{\partial{j_1}, \cdots,
 \partial{j_n}} f(X(h_1), \cdots,  X(h_m)) h_1 \otimes \cdots \otimes h_n
\end{align}

\noindent  Note when $n=1$, 
\begin{align}\label{a41}	
D F= \sum_{j=1}^{m} \frac{\partial}{\partial{j}} f(X(h_1), \cdots,  X(h_m)) h_j.
\end{align}	

\noindent  The $n$-th divergence operator  $\delta ^n :Dom(\delta^n) \to  L^2(\Omega)$
satisfy,  for each $u \in Dom(\delta^n)$,
	\begin{align*}	
E \Big( \big\langle D^n F, u   \big\rangle_{\mathcal{H}^{\otimes n}} 
	\Big)=E(F \delta^n(u)),~~\text{for ~all~}F \in \mathcal{S},
	\end{align*}	
where $Dom(\delta^n) \subset  L^2(\Omega, {\mathcal{H}^{\otimes n}})$ is defined by
\begin{align*}	
Dom(\delta^n)= \{ u \in L^2(\Omega, \mathcal{H}^{\otimes n}) | \langle D^nF, u  \rangle_{\mathcal{H}^{\otimes n}} \leq c \sqrt{E(F^2)}\},
\end{align*}
for some $c= c(u)>0$ and all $F \in \mathcal{S}$.
\noindent For more details and properties of $D^n$ and $\delta^n $, see NP (2012). 

\noindent In Section 3, the derivative and the divergence operators are acting 
on one dimensional deterministic functions, that is, for the case $\mathcal{H}=
\mathcal{R}$ , while in this section the domain of those operators is a collection
of random elements. However, we will continue to use the same notations and this
won't create a difficulty.

\noindent The next result gives the divergence of a Hilbert space valued rv.  If no confusion arises, we write sometimes $ \delta^n g(x)$ for  
$ (\delta^n g )(x)$.

\begin{lemma} \label{lem41} Let $ \mathcal{H}$ be a real 
separable	Hilbert space and $ X: \mathcal{H} \rightarrow L^2 (\Omega)$ be  an isonormal process. Let $ h \in \mathcal{H}$ be such that  $\left\| h \right\|=1$ and $ g: \mathcal{R} \to \mathcal{R} $ be a function so that $g(X(h))$ is a smooth rv. Then
\begin{align} \label{411a}
  \delta (g(X(h))h)=& (H-g)^{1}(X(h)),
\end{align}	
\end{lemma}
\noindent using the rule in  Remark \ref{rem31}.

\begin{proof}
Let $D$ be the Malliavin derivative defined in  \eqref{a41}. Recall that
$H_0(x)=1$ and $H_1(x)=x$.  Using  Proposition 2.5.4 of NP (2012), we have
\begin{align*}
\delta(g(X(h))h)= & g(X(h)) \delta(h)- \langle Dg(X(h)), h \rangle_{\mathcal{H}}\\
     = &  g(X(h)) X(h)-\langle g^{1}(X(h))h, h \rangle_{\mathcal{H}}\\
      =& g(X(h)) X(h)- g^{1}(X(h)) \langle h,  h \rangle_{\mathcal{H}}\\
      =& H_1(X(h)) g(X(h)) -H_0(X(h)) g^{1}(X(h))\\
      =& (H-g)^1(X(h)),
\end{align*} 
which proves the result. \end{proof}

\begin{remark} \label{rem41} \em
It follows from Theorem \ref{thm31} and  for $n=1$ that \eqref{411a} is equivalent to
\begin{align} \label{411b}
\delta (g(X(h))h)= (\delta^1 g) (X(h))=(\delta g) (X(h)),
\end{align}	
where $\delta^1=\delta $ on the rhs of \eqref{411b} is an operator on functions in the
sense of Section 3. Alternatively, $(\delta g)$ can also be viewed as an operator
on the isonormal process $ X= \{(X(h)| h \in \mathcal{H} \}$.
\end{remark}

\begin{lemma}\label{lem42}
Let, for $ n\geq 1$,  $ h_1, \ldots, h_n \in \mathcal{H}$ and $g \in \mathcal{S} $ be a smooth random variable. Then,
\begin{align} \label{412}
 \delta^{n+1} (g h_1 \otimes \cdots\otimes h_{n+1})= & \delta \Big(\delta^{n} (g h_1 \otimes \cdots\otimes h_{n})h_{n+1} \Big).
\end{align}
\end{lemma}

\begin{proof} Let $ f: R^m \rightarrow R$ be a smooth function
	whose derivatives have polynomial growth and $ F= f(X(v_1), \cdots, X(v_m)),
	v_1, \cdots,v_m \in \mathcal{H}$ be a smooth random variable.   For simplicity, let us denote
$h=  h_1 \otimes \cdots\otimes h_{n} \in H^{\otimes n}$. Then by the definition of
$\delta $, we have	
\begin{align*}
E\big[F \cdot \delta (\delta^n(gh) h_{n+1}) \big]=& E \langle DF, \delta^n(gh) h_{n+1} \rangle_{\mathcal{H}}  \\
=& E \Big\langle  \sum_{k=1}^{m} \frac{\partial f}{\partial x_k} \Big(X(v_1), \cdots, X(v_m)\Big) v_k,   ~\delta^n(gh) h_{n+1} \Big\rangle_{\mathcal{H}} \\ 
& \hspace{3cm} (\mbox{definition~ of~ Malliavin derivative~} D )  \\ 
=& \sum_{k=1}^{m} \langle v_k, h_{n+1} \rangle_{\mathcal{H}} E \Big\langle D^n \frac{\partial f}{\partial x_k} (X(v_1), \cdots, X(v_m)),  ~(gh) \Big\rangle_{\mathcal{H}^{\otimes n}}  \\
 & \hspace{3cm} (\mbox{using ~ the ~duality~ between}~
 \delta^{n}  \mbox{and }  D^{n} ) \\
=&E\Big\langle \sum_{k=1}^{m} D^n \frac{\partial f}{\partial x_k} (X(v_1), \cdots, X(v_m)) \otimes v_k, gh \otimes h_{n+1} \Big\rangle_{\mathcal{H}^{\otimes {n+1}}} \\ & \hspace{3cm} (\mbox{inner~ product ~ in~} \mathcal{H}^{\otimes {n+1}}) \\
=& E \Big \langle D^{n+1}F, gh \otimes h_{n+1} \Big \rangle_{\mathcal{H}^{\otimes {n+1}}} (\mbox{definition of Malliavin derivative }
  D^{n+1} )\\
=& E \Big[F \cdot~ \delta^{n+1} (gh \otimes h_{n+1})\Big] ~ (\mbox{duality~ between}~
 D^{n+1}~\mbox{and}~ \delta^{n+1}  ) \\
  \end{align*}
Thus, for each $ F \in \mathcal{S}$,  we obtain
\begin{align*}
E\big[F. ~ \delta (\delta^n(gh_1 \otimes \cdots \otimes h_n) h_{n+1}) \big]=& E \Big[F\cdot ~ \delta^{n+1} (gh_1 \otimes \cdots \otimes h_{n+1})\Big] 
\end{align*}
and the result follows.
\end{proof}

\vspace{.3cm}
\noindent The following result generalizes Theorem 2.7.7 of NP
(2012).

\begin{theorem} \label{thm41} Let $ \mathcal{H}$ be a real 
	separable	Hilbert space and $ X: \mathcal{H} \rightarrow L^2 (\Omega)$ be  an isonormal process. Let $ h \in \mathcal{H}$ be such that  $\left\| h \right\|=1$ and  $ g: {R} \to {R} $ be a function so that $g(X(h))$ is a smooth rv. Then, for $n \geq 1$,
	\begin{align} \label{413}
\delta^n(g(X(h))h^{\otimes n})=& (H-g)^n (X(h)),
	\end{align}
\end{theorem}
\noindent using the rule in  Remark \ref{rem31}.

\begin{proof}
The case $n=1$ is proved in Lemma \ref{lem41}. Assume now \eqref{413} holds so that
 	\begin{align*} 
 \delta^n(g(X(h))h^{\otimes n})=& (H-g)^n (X(h)).
 \end{align*}
By Theorem \ref{thm31}  and Remark \ref{411b},  the above assumption  is equivalent to
\begin{align} \label{415}
\delta^n(g(X(h))h^{\otimes n})=&  (\delta^n g) (X(h)),
\end{align}
where $\delta^n$ on the rhs of \eqref{415} is an operator on the set of functions, in the
sense of Section 3

\noindent  Consider now
 \begin{align*}
 \delta^{n+1}(g(X(h))h^{\otimes (n+1)})= &\delta \Big(\delta^{n}(g(X(h))h^{\otimes n})h\Big)
  \hspace{.2cm}(\text{using Lemma \ref{lem42}})\\
      = &\delta \Big((\delta^ng) (X(h))h \Big) \hspace{.2cm} \text{(by \eqref{415})}\\
     = &  \big(\delta (\delta^n g)\big)  (X(h) \big) \hspace{.3cm} \text{(using 
     \eqref{411b}~)}\\
     = &  (\delta^{n+1}g) (X(h)) \hspace{.4cm} \text{(using \eqref{312})} \\
     = &  (H-g)^{n+1}(X(h)),  
  \end{align*}
 using Theorem \ref{thm31}. This proves the result.
\end{proof}

\begin{remark} {\em Indeed, Theorem \ref{thm41} implies, for $n \geq 1$,
	\begin{align} \label{416b}
\delta^n(g(X(h))h^{\otimes n})=& (\delta^n g) (X(h)).
\end{align}	
}
\end{remark}	

\begin{example} {\em
\vspace{.3cm}
 (i) When $g(x)=1$, we get from  \eqref{413},
\begin{align*}
\delta^n(h^{\otimes n})=& (H-1)^n (X(h))\\
              =& H_n (X(h)),
\end{align*}
which is Theorem 2.7.7 of NP (2012).

\vspace{.3cm}
\noindent (ii) Similarly,  When $ g(x)=x$, we have
\begin{align*}
\delta^n (X(h)h^{\otimes n})=& (H-x)^n (X(h))\\
=& (H_1H_n - nH_{n-1 })(X(h))\\
=& H_{n+1}(X(h)),
\end{align*}
using \eqref{213}.

\vspace{.3cm}
(iii) When $n=2$, we obtain
\begin{align*}
\delta^{2}(g(X(h))h^{\otimes 2})= &(H-g)^2{(X(h))} \\
 =& (H_2g-2H_1g^1+ g^2)(X(h)) \\
  =& (X^2(h)-1)g(X(h))- 2 X(h) g^{1}(X(h))+ g^2(X(h)),
\end{align*}
using the properties of Hermite polynomials.
}
\end{example}

\begin{corollary} Assume the conditions of Theorem \ref{thm41} hold and let $D$ be the Malliavin derivative. Then
\begin{align} \label{416}
D(\delta^n(g(X(h))h^{\otimes n}))=& \Big(n(H-g)^{n-1}+ (H-g^1)^{n}\Big)(X(h)) h.
\end{align}	
\end{corollary}

\begin{proof} Note first for a real-valued function $f$ on  $R$, we have from \eqref{a41} that 
\begin{align} \label{ca41}
  D(f(X(h)))= f^{'}(X(h))h= (Df) (X(h))h.
\end{align}
Also, from Theorem \ref{thm41},
\begin{align*}
D(\delta^n(g(X(h))h^{\otimes n}))= & D \big((\delta^n g) (X(h)) \big)
\hspace{.3cm} \text{(using 
	\eqref{416b})}\\ 	
	 = &  (D(\delta^n g)) (X(h))h  \hspace{.3cm} \text{(using 
	 	\eqref{ca41})}\\ 
	 =&  [n \delta^{n-1}g + \delta^n g^1 ](X(h))h \hspace{.3cm}\text{(using ~Lemma \ref{lem32})} \\
     =& \big[n(H-g)^{n-1}+ (H-g^1)^{n}\big](X(h)) h,
\end{align*}
using Theorem \ref{thm31}. This proves the corollary.
\end{proof}

\noindent When $g(x)=1$, $g^1(x)=0$ and  $(H-0)^n=0$, using our rule
in Remark \ref{rem31}. Also, from \eqref{416},
\begin{align*}
D(\delta^n (h^{\otimes n}))
=& \big(n(H-1)^{n-1}+ (H-0)^{n}\big)(X(h))h \\
=& n H_{n-1} (X(h)) h,
\end{align*}
which is similar to Proposition 2.6.1 of NP (2012).

\noindent Similarly, when $g(x)=x$, we get
\begin{align*}
D(\delta^n (X(h) h^{\otimes n}))
=& \Big(n \big(H-x)^{n-1}+ (H-1)^{n}\Big)(X(h))h \\
=& \Big(n \big(H_{n-1}x-(n-1)H_{n-2} \big) + H_n\Big)(X(h))h \\
=& \Big(H_n+ nH_1 H_{n-1}-n(n-1)H_{n-2}  \Big) (X(h)) h,
\end{align*}
using $H_1=x.$

\nin  Note that the  formulas given above  are explicit and can easily be computed.

\vspace{.3cm}
\noindent {\bf Acknowledgments.}  This research was initiated when the second author was visiting the 
Depart of Statistics and Probability, Michigan State University, during Summer-2019. The authors thank  Dr. Frederi Viens for his support and
encouragements.

\vspace{.4cm}
\end{document}